\documentclass{article}
\usepackage[leqno]{amsmath}
\usepackage{amsmath,amsfonts,amsthm,amssymb}
\usepackage{epsfig}
\usepackage{graphicx}
\usepackage{subfig}

\usepackage{cite}
\usepackage{afterpage,hyperref}
\hypersetup{
  colorlinks,
  citecolor=green!50!black,
  linkcolor=red,
  urlcolor=blue!50!black}
\usepackage{tikz}
\usetikzlibrary{backgrounds,calc,positioning}
\usepackage{epsfig}
\usepackage{color}
\usepackage{amsmath}
\usepackage{amssymb}
\newtheorem{lemma}{Lemma}
\newtheorem{proposition}{Proposition}

\newtheorem{observation}{Observation}

\theoremstyle{definition}

\newtheorem{conjecture}{Conjecture}

\newtheorem{con}{Conjecture}

\usetikzlibrary{decorations.markings} 

\tikzstyle{graphnode}=[draw,shape=circle,draw=black,minimum size=0.5pt,inner sep=1.5pt]

\title{Cuts in matchings of 3-connected cubic graphs}
\author{Kolja Knauer\thanks{Aix Marseille Univ, Universit\'e de Toulon, CNRS, LIS, Marseille, France.\newline
\indent Email: \texttt{\{kolja.knauer,petru.valicov\}@lis-lab.fr}}
\and Petru Valicov\footnotemark[1]}

\begin{document}
\maketitle
\begin{abstract}
We discuss conjectures on Hamiltonicity in cubic graphs (Tait, Barnette, Tutte), on the dichromatic number of planar oriented graphs (Neumann-Lara), and on even graphs in digraphs whose contraction is strongly connected (Hochst\"attler). We show that all of them fit into the same framework related to cuts in matchings. This allows us to find a counterexample to the conjecture of Hochst\"attler and
 show that the conjecture of Neumann-Lara holds for all planar graphs on at most 26 vertices. Finally, we state a new conjecture on bipartite cubic oriented graphs, that naturally arises in this setting.
\end{abstract}

\section{Introduction}
Let us first introduce three conjectures on 3-connected, cubic graphs in their order of appearance.

\begin{conjecture}[Tait 1884]
\label{conj:Tait}
Every 3-connected, cubic, planar graph contains a Hamiltonian cycle.
\end{conjecture}

The first counterexample to Conjecture~\ref{conj:Tait} was given by Tutte~\cite{T46} and is a graph on 46 vertices. Several smaller counterexamples on 38 vertices were later found by Holton and McKay, who also proved that there is no counterexample with less than 38 vertices~\cite{HMcK88}. 
One can observe that all known counterexamples to Tait's conjecture have odd cycles. Maybe this is essential:

\begin{conjecture}[Barnette 1969~\cite{B69}]
\label{conj:Barnette}
Every 3-connected, cubic, planar, bipartite graph contains a Hamiltonian cycle.
\end{conjecture}

In general Conjecture~\ref{conj:Barnette} remains open. It was shown to be true for graphs with at most 66 vertices~\cite{HMMcK85}. There is an announcement~\cite{MoharWeb} claiming that the conjecture was verified for graphs with less than 86 vertices. A few years later a stronger conjecture was proposed:

\begin{conjecture}[Tutte 1971~\cite{T71}]
\label{conj:Tutte}
Every 3-connected, cubic, bipartite graph contains a Hamiltonian cycle.
\end{conjecture}

Conjecture~\ref{conj:Tutte} was disproved by Horton~\cite{H82}. The smallest known counterexample has 50 vertices and was discovered independently by Georges~\cite{G89} and Kelmans~\cite{K88}. Moreover, in~\cite{K88} it is claimed (but no reference given) that Lomonosov and Kelmans proved Conjecture~\ref{conj:Tutte} for graphs on at most 30 vertices. We verified Conjecture~\ref{conj:Tutte} on graphs up to 40 vertices by computer.

Let us now go back to the plane and introduce a seemingly unrelated conjecture in digraphs. The \emph{digirth} of a digraph $D=(V,A)$ is the length of a shortest directed cycle. A digraph is called \emph{oriented graph} if it is of digirth at least $3$. A set of vertices $V'\subseteq V$ is acyclic in $D$ if the digraph induced by $V'$ contains no directed cycle. Neumann-Lara stated the following:

\begin{conjecture}[Neumann-Lara 1985~\cite{NL85}]
\label{conj:Neumann-Lara}
Every planar oriented graph can be vertex-partitioned into two acyclic sets.
\end{conjecture}

Conjecture~\ref{conj:Neumann-Lara} remains open in general, but was recently proved for graphs with digirth at least 4~\cite{LM17}. Here we give the first computational evidence for it by showing that it is valid for all planar graphs on at most 26 vertices, see Proposition~\ref{prop:NeumannLara26}. 

Let us introduce another seemingly unrelated problem. Given a (partially) directed graph $D=(V,A)$, for $E\subseteq A$, let $D / E$ denote the graph obtained from $D$ by contracting the edges of $E$. An \emph{even} subgraph $E$ of a digraph $D=(V,A)$ is a subset $E\subseteq A$ that is an edge-disjoint union of cycles of $D$ (the cycles are not necessarily directed). 
Recently, Hochst\"attler proposed:

\begin{conjecture}[Hochst\"attler 2017~\cite{H17}]
\label{conj:Hochstattler}
In every 3-edge-connected digraph $D=(V,A)$ there exists an even subgraph $E\subseteq A$ such that $D/E$ is strongly connected.
\end{conjecture}

We construct a counterexample to Conjecture~\ref{conj:Hochstattler} on $122$ vertices, see Proposition~\ref{prop:Hochstattlercounter}. 

\bigskip

In Section~\ref{sec:alltogether} we explain that all these conjectures arise in the same context and naturally lead to a new question (Conjecture~\ref{conj:KV}). In Section~\ref{sec:construction} we give a general method that might be helpful to search for counterexamples to Conjectures~\ref{conj:Neumann-Lara} and~\ref{conj:KV}. In particular, this method leads to a counterexample to Conjecture~\ref{conj:Hochstattler}. We explain our computational results in Section~\ref{sec:computation} and conclude the paper in Section~\ref{sec:concl}.

\section{Cubic graphs without perfect matchings containing a cut}\label{sec:alltogether}
In the present section we will reformulate all of the above conjectures as statements about cubic (planar, bipartite, directed) graphs with some perfect matchings not containing a (directed) cut. See Figure~\ref{fig:diagram} for an illustration of the relations between the conjectures discussed in this section.

\begin{figure}[!ht]
\hypersetup{linkcolor=blue}
\centering
\begin{tikzpicture}[
block/.style={
draw,
circle, 
minimum width={width("Hochstattl")},
align=center,
font=\scriptsize}, join=bevel,inner sep=0.5mm]
\node[draw,fill=blue, shape=circle,draw=black,minimum size=1pt,inner sep=2.5pt](000) at (0,0) {};
\node[block, color=blue](001) at (3,2.5) {Neumann\\-Lara};
\node[block, color=blue](010) at (0,2.5) {Conjecture~\textcolor{blue}{\ref{conj:KV}}};
\node[block, color=blue](100) at (-3,2.5) {Barnette};

\node[block, color=red!80!black](110) at (-3,5) {Tutte};
\node[block, color=red!80!black](101) at (0,5) {Tait};
\node[block, color=red!80!black](011) at (3,5) {Hochst\"attler};
\node[draw,shape=circle,draw=black,minimum size=1pt,inner sep=2.5pt, fill=red!80!black](111) at (0,7.5) {};

\draw[-]  (000)--(001) (001)--(011) (011)--(010)--(000) (100)--(101)--(111);
\draw[-]  (100)--(110)--(111)--(101)--(100) (001)--(000);
\foreach \I in {00,01,10,11}
\draw[-] (1\I)--(0\I);

\node at (-1.6,6.5) {\textbf{b}};
\node at (0.2,6.5) {\textbf{p}};
\node at (1.6,6.5) {\textbf{d}};
\end{tikzpicture}
\hypersetup{linkcolor=red}
\caption{Conjectures of the form "\textit{Every 3-connected, cubic, \textbf{b}ipartite, \textbf{p}lanar, \textbf{d}irected graph contains a perfect matching without (directed) cut}". The upper four are false, the lower four are open.}
\label{fig:diagram}
\end{figure}
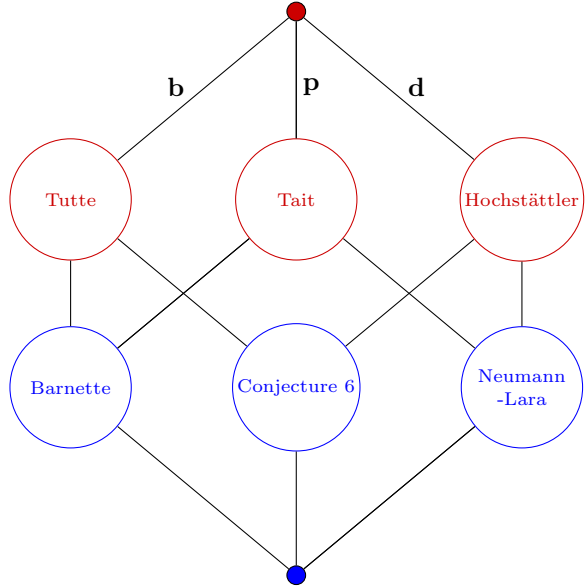

First, observe that a cubic graph $G=(V,E)$ contains a Hamiltonian cycle ${C}$, if and only if the complement $E\setminus {C}$ is a perfect matching containing no edge-cut. Therefore, in Conjectures~\ref{conj:Tait}, \ref{conj:Barnette}, \ref{conj:Tutte} the property \emph{contains a Hamiltonian cycle} can be replaced equivalently by \emph{contains a perfect matching without cut}.
This section is dedicated to the proof that Conjectures~\ref{conj:Neumann-Lara} and~\ref{conj:Hochstattler} can be  reformulated equivalently as below, which justifies their placement in Figure~\ref{fig:diagram}.

\begin{con}[Neumann-Lara]
\label{conj:Neumann-Lara'}
 Every 3-connected, cubic, planar digraph  contains a perfect matching without a directed cut.
\end{con}

\begin{con}[Hochst\"attler]
\label{conj:Hochstattler'}
 Every 3-edge-connected, cubic digraph contains a perfect matching without a directed cut.
\end{con}

Before proceeding to the proof, note that a new question follows naturally:

\begin{conjecture}
\label{conj:KV}
 Every 3-connected, cubic, bipartite digraph contains a perfect matching without directed cut.
\end{conjecture}

Let us now proceed to the proof, which we have split into several lemmas.
\begin{lemma}\label{lem:duality}
 Conjecture~\ref{conj:Neumann-Lara} is equivalent to Conjecture~\ref{conj:Hochstattler} restricted to planar graphs.
\end{lemma}
\begin{proof}
Partitioning the vertex set of an oriented graph $D$ in two acyclic sets $X,Y$ corresponds to finding an edge-cut $C$ (consisting of the edges between $X$ and $Y$) such that $D\setminus C$ is acyclic. If $D$ is planar, we can look at the dual digraph $D^*$, where an edge of $D^*$ (corresponding to an arc $a$ of $D$) is oriented from the face on the left of $a$ to the face on the right of $a$. It is well-known that under this duality, we have the following:
\begin{itemize}
 \item $D$ is strongly connected $\iff$ $D^*$ is acyclic,
 \item $D$ is simple  $\iff$ $D^*$ is $3$-edge-connected,
 \item $C$ is a cut in $D$  $\iff$ the corresponding arcs $E$ in $D^*$ form an even subgraph,
 \item for any set $A'$ of arcs in $D$ and corresponding arcs $B'$ in $D^*$, we have $(D\setminus A')^*=D^*/B'$.
\end{itemize}
Since $D$ has digirth at least $3$ it can be assumed to be simple. So with the above properties we get that in a planar oriented graph $D$ there is a cut $C$ such that $D\setminus C$ is acyclic, if and only if in the $3$-edge-connected digraph $D^*$ there is an even subgraph $E$ such that $D^*/E$ is strongly connected. Thus, Conjecture~\ref{conj:Neumann-Lara} can be reformulated as: \emph{every 3-edge-connected, planar digraph $D=(V,A)$ contains an even subgraph $E\subseteq A$ such that $D/E$ is strongly connected}.
\end{proof}

Define an \emph{odd subgraph} $O\subseteq A$ of a digraph $D=(V,A)$ such that every vertex has odd degree with respect to $O$.

\begin{lemma}\label{lem:cubic}
Conjecture~\ref{conj:Hochstattler} (restricted to planar graphs) is equivalent to:
Every (planar,) 3-connected, cubic digraph $D=(V,A)$ contains an odd subgraph without directed cut.
\end{lemma}
\begin{proof}
First, we show that both conjectures are equivalent to their restriction to cubic graphs. One direction is trivial i.e. if Conjecture~\ref{conj:Hochstattler} is true for general (planar) graphs, then in particular it holds for cubic (planar) graphs. For the other direction suppose Conjecture~\ref{conj:Hochstattler} for (planar) cubic graphs and let $D$ be a 3-edge-connected directed graph $D=(V,A)$. Replace every vertex $v$ of $D$ with an undirected cycle $C_v$ of length $\mathrm{deg}(v)$ such that the resulting partially directed graph $D'=(V',A')$ is cubic and 3-connected. Let $E'\subseteq A'$ be an even subgraph such that $D'/E'$ is strongly connected. Since $D$ is a contraction of $D'$ and strong connectivity is closed under contraction, then $D/(E'\cap A)$ is strongly connected. Since $D'$ is cubic, then $E'$ is a disjoint union of cycles $C_1,\ldots, C_k$. Thus, for any of the cycles of the form $C_v$ in $D'$, either $C_v\in\{C_1,\ldots, C_k\}$ or $C_v$ intersects $E'$ in a set of paths. Therefore, after contracting $C_v$ back, vertex $v$ has even degree with respect to $E'\cap A$. Thus, $E'\cap A$ is an even subgraph and we are done. Note that this construction preserves planarity and 3-connectivity.

Finally, observe that a cubic digraph $D=(V,A)$ contains an even subgraph such that $D/E$ is strongly connected if and only if the complement $O=A\setminus E$ is an odd subgraph containing no directed edge-cut.  
\end{proof}


Given a cubic digraph $D=(V,A)$ define the cubic digraph $D^+=(V',A')$ by replacing each vertex $v$ of $D$ that is neither a sink nor a source by $7$ vertices as shown in Figure~\ref{fig:replacement}. The preserved arcs from $D$ (depicted in bold) are called \emph{suspension arcs}.

\begin{figure}[!ht]
\centering
\begin{tikzpicture}[join=bevel,inner sep=0.5mm]

\node[graphnode](vertex) at (-7,1) {$v$};
\node[graphnode](u) at (-7,-1) {$u$};
\node[graphnode](w) at (-5,2.5) {$w$};
\node[graphnode](x) at (-9,2.5) {$x$};

\draw[->,line width=2pt] (vertex)--(u);
\draw[<-,line width=2pt] (vertex)--(w);
\draw[<-,line width=2pt] (vertex)--(x);

\node[graphnode](a) at (0,0) {$v_1$};
\node[graphnode](b) at (1,1) {$v_2$};
\node[graphnode](c) at (1,2) {$v_3$};
\node[graphnode](d) at (0,3) {$v_4$};
\node[graphnode](e) at (-1,2) {$v_5$};
\node[graphnode](f) at (-1,1) {$v_6$};
\node[graphnode](g) at (0,1.75) {$v_7$};

\node[graphnode](u1) at (0,-1) {$u$};
\node[graphnode](w1) at (2,2.5) {$w$};
\node[graphnode](x1) at (-2,2.5) {$x$};

\draw[->,line width=2pt] (a)--(u1);
\draw[<-,line width=2pt] (c)--(w1);
\draw[<-,line width=2pt] (e)--(x1);
\draw[postaction=decorate, decoration={markings, mark=at position 0.6 with{\arrow[line width=1pt]{stealth}}}] (g)--(f);
\draw[postaction=decorate, decoration={markings, mark=at position 0.6 with{\arrow[line width=1pt]{stealth}}}] (g)--(b);
\draw[postaction=decorate, decoration={markings, mark=at position 0.6 with{\arrow[line width=1pt]{stealth}}}] (g)--(d);
\draw[postaction=decorate, decoration={markings, mark=at position 0.6 with{\arrow[line width=1pt]{stealth}}}] (d)--(e);
\draw[postaction=decorate, decoration={markings, mark=at position 0.6 with{\arrow[line width=1pt]{stealth}}}] (f)--(e);

\draw[postaction=decorate, decoration={markings, mark=at position 0.6 with{\arrow[line width=1pt]{stealth}}}] (a)--(f);
\draw[postaction=decorate, decoration={markings, mark=at position 0.6 with{\arrow[line width=1pt]{stealth}}}] (a)--(b);
\draw[postaction=decorate, decoration={markings, mark=at position 0.6 with{\arrow[line width=1pt]{stealth}}}] (b)--(c);
\draw[postaction=decorate, decoration={markings, mark=at position 0.6 with{\arrow[line width=1pt]{stealth}}}] (d)--(c);
\end{tikzpicture}
\caption{The replacement of a vertex $v$ of $D$ by $R_v$ to obtain $D^+$. The bold arcs are the arcs which are preserved (suspension arcs).}
\label{fig:replacement}
\end{figure}
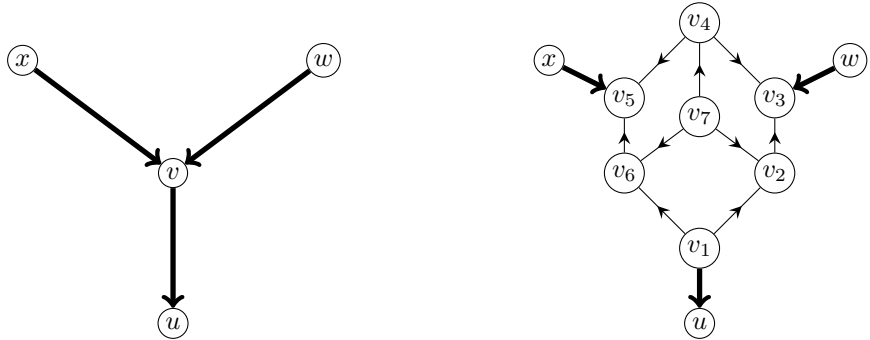

\begin{lemma}\label{lem:end}
Let $D=(V,A)$ be a cubic digraph such that every perfect matching contains a directed cut. Every odd subgraph of $D^+=(V',A')$ contains a directed cut. Furthermore, $|V'|=|V|+6|\{v\in V\mid \delta^+(v)\in\{1,2\}\}|$ and if $D$ is 3-connected, bipartite or planar, then so is $D^+$.
\end{lemma}
\begin{proof}
Suppose that every perfect matching of $D$ contains a directed cut and let $O\subseteq A'$ be an odd subgraph of $D^+$. Since $D^+$ is cubic, all vertices have degree 1 or 3 in $O$.

Consider some replacement gadget $R_v$ as in Figure~\ref{fig:replacement}. 
Observe that $R_v$ has an odd number of vertices $v_1,\ldots,v_7$ and every odd subgraph has an even number of vertices. Therefore, among the arcs $(w,v_3)$, $(x,v_5)$, $(v_1,u)$, an odd number has to be in $O$, since otherwise an odd number of vertices remains to be covered in $R_v$. We distinguish two cases:

Suppose $(w,v_3)$, $(x,v_5)$, $(v_1,u)$ are all in $O$. Then $v_1$, $v_3$ and $v_5$ must be of degree 1 in $O$, since otherwise they have degree 3 and we get a sink or a source in $O$. Hence, vertices $v_2,v_4,v_6,v_7$ remain to be covered by arcs of $O$ and thus $v_7$ necessarily has degree 3 in $O$ and is a source in $O$. Thus, $O$ contains a directed cut.

The remaining case is when for every replacement gadget $R_v$ of $G$, exactly one suspension arc is in $O$. In this case, by construction of $D^+$, all $O$ induces a perfect matching of $D$. By assumption, this perfect matching contains a directed cut, which also is a directed cut of $D^+$ and we are done.

It is easy to see that the construction preserves 3-connectivity, bipartiteness, and planarity and has the claimed number of vertices.
\end{proof}

Putting Lemmas~\ref{lem:duality},~\ref{lem:cubic}, and~\ref{lem:end} together we get
\begin{proposition}
 Conjectures~\ref{conj:Neumann-Lara} and~\ref{conj:Hochstattler} are equivalent to Conjectures~\ref{conj:Neumann-Lara'} and~\ref{conj:Hochstattler'}, respectively.
\end{proposition}

Note that the statement corresponding to the top of Figure~\ref{fig:diagram} is that every cubic, 3-connected graph is Hamiltonian. The first counterexample to this statement is the Petersen graph. While it is neither bipartite, nor planar, all of its orientations satisfy Conjecture~\ref{conj:Hochstattler}, that is, every orientation of the Petersen graph contains an odd subgraph without directed cut~\cite{H17}. By computer we verified that each of these orientations even has a perfect matching containing no directed cut, i.e. the orientations of the Petersen graph satisfy Conjecture~\ref{conj:Hochstattler'}. However, we will see in the next section that the Petersen graph plays a crucial role for the counterexamples to Conjectures~\ref{conj:Hochstattler} and~\ref{conj:Hochstattler'}.

\section{Counterexamples to Conjectures~\ref{conj:Hochstattler} and~\ref{conj:Hochstattler'} from $a$-arcs}
\label{sec:construction}

An edge $\{u,v\}$ in a 3-connected, cubic graph $G$ is called an \emph{a-edge}, if every perfect matching of $G$ that contains $\{u,v\}$, also contains a cut. Equivalently, every Hamiltonian cycle of $G$ must pass through this edge. The name $a$-edge was defined in~\cite{B66} (also used in~\cite{HMcK88}). The similar notion can be defined for directed graphs: an arc $(u,v)$ in a partially directed, 3-connected, cubic graph $D$ is called an \emph{a-arc} if every perfect matching of $D$ that contains $(u,v)$ also contains a directed cut. Clearly an $a$-arc $(u,v)$ of a directed graph corresponds to the $a$-edge $\{u,v\}$ of the underlying undirected graph.
For an $a$-arc $(u,v)$ we denote by $v',v''$ the other two neighbors of vertex $v$ in $D$. Since the edges $\{v,v'\}, \{v,v''\}$ are not contained in any perfect matching containing $\{u,v\}$, we can assume that  $\{v,v'\}, \{v,v''\}$ are not directed in $D$.

Given a partially directed, 3-connected, cubic graph $D$ with a specified $a$-arc $(u,v)$, we construct two partially directed, 3-connected, cubic graphs $\hat{D}$ and $\widetilde{D}$, both having the property that every perfect matching contains a directed cut. The construction $\hat{D}$ yields the smallest counterexample to Conjecture~\ref{conj:Hochstattler'} (see Figure~\ref{subfig:smallestcount}), whereas $\widetilde{D}$ yields the smallest known counterexample to Conjecture~\ref{conj:Hochstattler}. Moreover, $\widetilde{D}$ preserves bipartiteness and planarity of $D$.
For both constructions we first build the graph $
D'$ by splitting vertex $u$ into $u,u',u''$ such that each has exactly the corresponding neighbor among $v, v',v''$, see Figure~\ref{subfig:a-arcs-constructions1}

\begin{figure}[!ht]
 \centering
 \subfloat[the graph $D'$]{\label{subfig:a-arcs-constructions1}\includegraphics[scale=1]{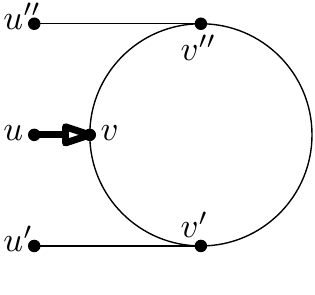}}\hspace*{0.4cm}
 \subfloat[the graph $\hat{D}$; red and blue arcs form directed cuts, respectively.]{\label{subfig:a-arcs-constructions2}\includegraphics[scale=1]{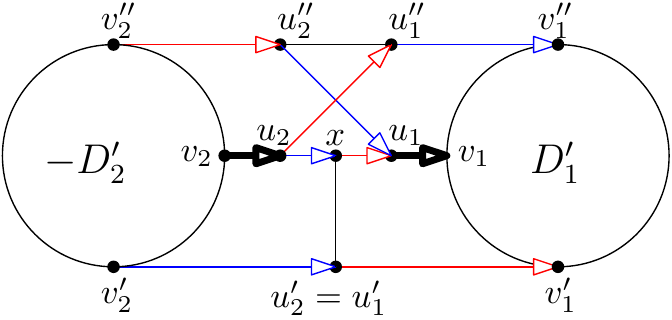}}
  \subfloat[the graph $\widetilde{D}$, $y=u_1=u_2=u_3$]{\label{subfig:a-arcs-constructions3}\includegraphics[scale=1]{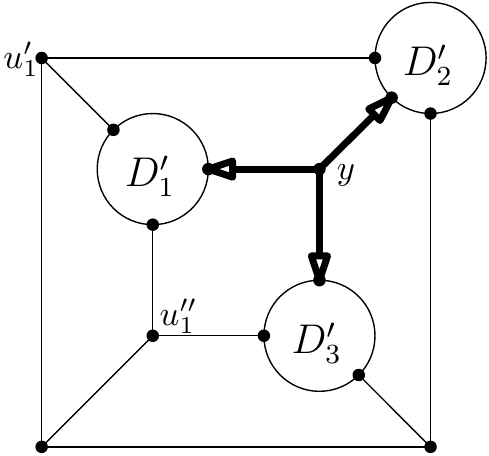}}
 \caption{Counterexamples from $a$-arcs}
\label{fig:constructions}
\end{figure}

We build the graph $\hat{D}$ from two copies of $D'$: $D'_1$ (isomorphic to $D'$) and an oppositely oriented copy $-D'_2$. We connect these copies as shown in Figure~\ref{subfig:a-arcs-constructions2}. We call this graph $\hat{D}$. 

\begin{lemma}
\label{lem:allPMhaveACut'}
Every perfect matching of $\hat{D}$ contains a directed cut.
\end{lemma}
\begin{proof}
Let $M$ be a perfect matching of $\hat{D}$.
Note that by the parity of $D'_1$, $M$ contains either all the three arcs $(u_1,v_1),(u'_1,v'_1),(u''_1,v''_1)$ or exactly one of them.  However, if it was all three, this would yield a directed cut in $M$.  Symmetrically, $M$ contains exactly one arc among $(v_2,u_2),(v'_2,u'_2),(v''_2,u''_2)$.
Thus, we have one of the following three situations, where in each case we suppose that none of the previous ones apply:
\begin{itemize}
\item $M$ contains $(u_1,v_1)$ or $(v_2,u_2)$. Since each of these arcs corresponds to an $a$-arc of $D$, the restriction $M'$ of $M$ to the corresponding copy of $D'$ contains a directed cut of $D'$. This directed cut is contained in $M$ and is a directed cut of $\hat{D}$.
\item $M$ contains the arc $(v_2',u_2')$. Then $M$ also contains $(u''_1,v''_1)$ and one deduces that $M$ contains the blue directed cut.
\item $M$ contains the arc $(v_2'',u_2'')$. Then $M$ also contains $(x,u_1)$ and one deduces that $M$ contains the red directed cut.
\end{itemize}
\end{proof}


\begin{figure}[!ht]
 \centering
 \subfloat[partial orientation $P$ of the Petersen graph with the two matchings containing arc $(u,v)$.]{\label{subfig:petersen}\includegraphics[scale=0.82]{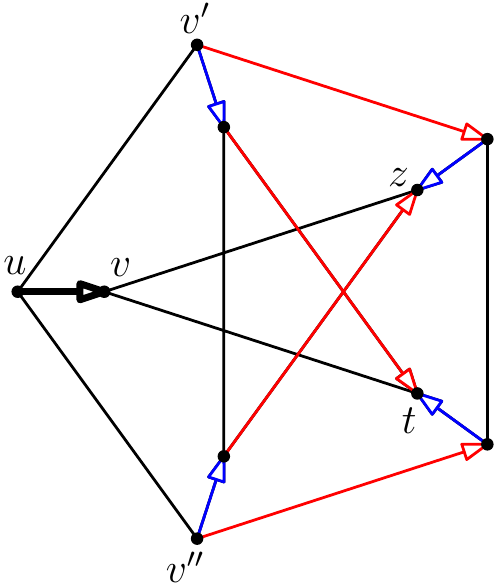}}\hspace*{0.5cm}
 \subfloat[smallest cubic, 3-connected, partially directed graph where each perfect matching contains a directed cut.]{\label{subfig:smallestcount}\includegraphics[scale=0.82]{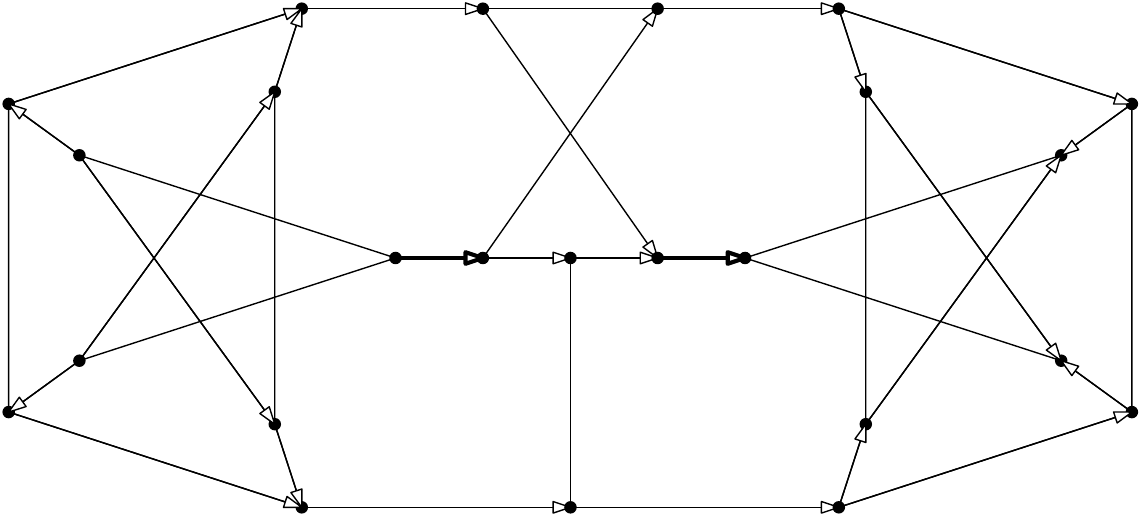}}
 \caption{The construction of the smallest counterexample to Conjecture~\ref{conj:Hochstattler'}}
 \label{fig:counterexample}
\end{figure}

Let us now use Lemma~\ref{lem:allPMhaveACut'} to give the smallest counterexample to Conjecture~\ref{conj:Hochstattler'}. 

\begin{observation}
\label{obs:petersenDirectedCuts}
In the partial orientation $P$ of the Petersen graph given in Figure~\ref{subfig:petersen}, each of the (two) perfect matchings with arc $(u,v)$ contains a directed cut, i.e., $(u,v)$ is an $a$-arc.
\end{observation}

The corresponding graph $\hat{P}$ coming from Lemma~\ref{lem:allPMhaveACut'} has $24$ vertices and is depicted in Figure~\ref{subfig:smallestcount}. Together, with our computational results in Section~\ref{sec:computation} we get:
\begin{proposition}
 The graph of Figure~\ref{subfig:smallestcount} is the smallest counterexample to Conjecture~\ref{conj:Hochstattler'}.
\end{proposition}

Let us now describe the second construction $\widetilde{D}$, which will lead to the smallest known counterexample to Conjecture~\ref{conj:Hochstattler}. Let again $D$ and $D'$ be as in the beginning of the section.  We build $\widetilde{D}$ by replacing three vertices of the cube-graph $Q_3$ by copies $D'_1,D'_2,D'_3$ of $D'$ and by identifying vertices $u_1,u_2,u_3$  to the vertex $y$ as shown in Figure~\ref{subfig:a-arcs-constructions3}. The graph $\widetilde{D}$ is 3-connected, cubic and the construction preserves planarity and bipartiteness of $D$.

\begin{lemma}
\label{lem:allPMhaveACut}
Every perfect matching of $\widetilde{D}$ contains a directed cut.
\end{lemma}
\begin{proof}
Let $M$ be a perfect matching of $\widetilde{D}$. Exactly one of the arcs incident to vertex $y$ is in $M$. Say this is the arc to $D'_1$. By parity of $D'_1$, either both edges connecting $u'_1$ and $u''_1$ to $D'_1$ are in $M$ or none of them. If both are in $M$, then it is easy to see, that $M$ cannot be a perfect matching of $\widetilde{D}$. Therefore, $M$ contains a perfect matchings $M'$ of $D'_1$ that corresponds to a perfect matching of $D$. Since $(u,v)$ is an $a$-arc of $D$, $M'$ contains a directed cut of $D'$. This directed cut is contained in $M$ and is a directed cut of $\widetilde{D}$.
\end{proof}

Again, let $P$ be the same orientation of the Petersen graph depicted in Figure~\ref{subfig:petersen}, which by Observation~\ref{obs:petersenDirectedCuts} has an $a$-arc. With $32$ vertices, the graph $\widetilde{P}$ is a larger counterexample to Conjecture~\ref{conj:Hochstattler'} than $\hat{P}$. However, by means of Lemma~\ref{lem:end} the graph $\widetilde{P}$ can be used to construct a relatively small counterexample to Conjecture~\ref{conj:Hochstattler}. More precisely, we fully orient $\widetilde{P}$, such that in every copy of $P$, vertices $v', v''$ are sources and vertices $z,t$ are sinks. Finally, the remaining non-oriented edges of the cube $Q_3$ are oriented so that all of the vertices of $Q_3$ (except the ones where a copy of $P$ was inserted) are sinks or sources. Since in Lemma~\ref{lem:end} only vertices that are neither sinks nor sources have to be replaced by a gadget of $7$ vertices, this results in a cubic counterexample to Conjecture~\ref{conj:Hochstattler} of order $32+6\cdot 5\cdot 3=122$. We obtain:

\begin{proposition}\label{prop:Hochstattlercounter}
There is a counterexample to Conjecture~\ref{conj:Hochstattler} on 122 vertices.                                                                                                                                                                                                                                                                                                                                                                                                                                                                                                                                                                                                                                                                                                                                                                                                                                              \end{proposition}

Note that any full orientation of the graph of Figure~\ref{subfig:smallestcount} has at most 4 vertices which are sinks or sources. Therefore, if we apply Lemma~\ref{lem:end} to this graph, the smallest counterexample to Conjecture~\ref{conj:Hochstattler} we can get has order $24+20\cdot 6 = 144$

\section{Searching for counterexamples}\label{sec:computation}
In this section we first give some easy observations for reducing the search space, then we explain how we checked the conjectures on large sets of digraphs computationally. Afterwards, we summarize our computational results.

\subsection{Properties of minimum counterexamples}
It is well known, that contracting a triangle in a counterexample to Conjecture~\ref{conj:Tait} yields a smaller counterexample. The same is easily seen to hold for Conjectures~\ref{conj:Neumann-Lara'} and~\ref{conj:Hochstattler'}. We therefore have:

\begin{lemma}~\label{lem:girth4}
Let $G$ be a cubic (partially directed) graph with a triangle $\Delta$ and let $G_{\Delta}$ be the graph obtained from $G$ by contracting $\Delta$ into a single vertex.

Then the following holds:
\begin{enumerate}
\item If $G$ is a counterexample to Conjectures~\ref{conj:Tait}, ~\ref{conj:Neumann-Lara'} or~\ref{conj:Hochstattler'}, then $G_{\Delta}$ is also a counterexample.
\item If $G$ is $3$-connected and has an $a$-edge ($a$-arc), then $G_{\Delta}$ has an $a$-edge ($a$-arc).
\end{enumerate}
\end{lemma}

\begin{figure}[!ht]
\centering
\subfloat[Constructing $G_{u,v}$ from $G$]{\label{subfig:tutte1}\includegraphics[scale=1]{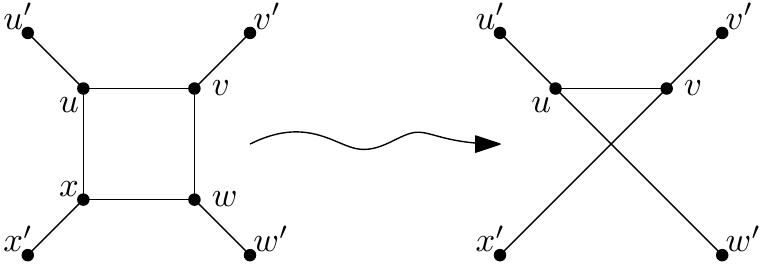}}\hspace*{2cm}
 \subfloat[hypothetical placement of $e$ and $e'$ in  Lemma~\ref{lem:girth6}.]{\label{subfig:tutte2}\includegraphics[scale=1]{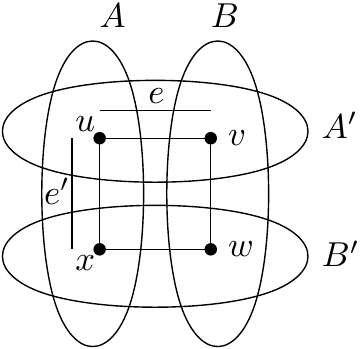}}
 \caption{Cases described in the proof of Lemma~\ref{lem:girth6}}
 \label{fig:tutte}
\end{figure}

A similar result holds for Conjecture~\ref{conj:Tutte}. 
\begin{lemma}~\label{lem:girth6}
 Any minimum counterexample to Conjecture~\ref{conj:Tutte} has girth at least $6$.
\end{lemma}
\begin{proof}
 Let $G$ be a $3$-connected, cubic, bipartite graph and $C=(u,v,w,x)$ a $4$-cycle in $G$. Denote by $u',v',w',x'$ the neighbors of $u,v,w,x$ outside $C$, respectively. Define $G_{u,v}$ to be the graph obtained from the vertex-deleted graph $G\setminus\{w,x\}$ by adding the edges $\{x',v\}$ and $\{w',u\}$ (see Figure~\ref{subfig:tutte1}). Similarly, define $G_{v,w}$. Both $G_{u,v}$ and $G_{v,w}$ are bipartite and cubic. A little case distinction easily yields, that if $G_{u,v}$ or $G_{v,w}$ has a Hamiltonian cycle, then so does $G$. To show this one has to consider how vertices $u',u,v,v',x',w'$ are used in a Hamiltonian cycle of one of these two graphs (say $G_{u,v}$). For example, consider the case when both paths $u'uw'$ and $x'vv'$ are used in a Hamiltonian cycle in $G_{u,v}$. Then either there exists a path $P_1$ from $x'$ to $w'$ disjoint with  $\{u',v',u,v\}$ (and consequently there is a similar path $P'_1$ from $v'$ to $u'$); or there exists a path $P_2$ from $x'$ to $u'$ disjoint with  $\{u,v,v',w'\}$ (and consequently there is a similar path $P'_2$ from $w'$ to $v'$). For the first subcase, the Hamiltonian cycle in $G$ is $u'uxx'P_1w'wvv'P'_1u'$. The second subcase is symmetric, the Hamiltonian cycle in $G$ being $u'P_2x'xww'P'_2v'vuu'$.

Finally, let us see that at least one of $G_{u,v}$, $G_{v,w}$ is $3$-connected. Indeed, if there is a $2$-edge cut in $G_{u,v}$, then there is a $3$-edge cut in $G$ containing $\{u,v\}$ and $\{w,x\}$ and some edge $e\notin C$, splitting $G$ in $A\ni u,x$ and $B\ni v,w$. Similarly, if there was a $2$-edge cut in $G_{v,w}$ there is a $3$-edge cut in $G$ with an edge $e'\notin C$ splitting $G$ in $A'\ni u,v$ and $B'\ni w,x$. Depending on which of $A', B'$ contains $e$ and which of $A, B$ contains $e'$, one of the cuts induced by $A\cap A'$, $A'\cap B$, $B\cap B'$, $B'\cap A$, consists of only two edges (see Figure~\ref{subfig:tutte2}). This contradicts the $3$-connectivity of $G$.

Therefore, if $G$ is a counterexample to Conjecture~\ref{conj:Tutte}, then at least one of the graphs $G_{u,v}$ or $G_{v,w}$ is also a counterexample. Hence $G$ cannot be a minimum counterexample.
\end{proof}

\subsection{Computational results for Conjecture~\ref{conj:Tutte}}
By Lemma~\ref{lem:girth6}, in order to verify Conjecture~\ref{conj:Tutte} for graphs of small order, we only have to check Hamiltonicity for cubic, bipartite, $3$-connected graphs of girth at least $6$. All cubic bipartite graphs on at most 34 vertices are available on-line (\url{https://hog.grinvin.org/Cubic}). The list was generated using Minibaum, a computer program written by Brinkmann~\cite{B96}. We also used Minibaum to generate bipartite, cubic graphs of girth at least 6 on at most 40 vertices. We verified that all 3-connected graphs among these are Hamiltonian by launching the Hamiltonicity test of SageMath~\cite{Sage} on each of them. Together with Lemma~\ref{lem:girth6}, this confirms that all bipartite, cubic, 3-connected graphs on at most 40 vertices are Hamiltonian. 

\subsection{Computational results for Conjectures~\ref{conj:Neumann-Lara'} and~\ref{conj:Hochstattler'}}

First, recall that for the cases of directed graphs, the instances of undirected graphs we are considering are always non-Hamiltonian, that is every perfect matching contains a cut.

In order to verify that for any orientation of a given undirected cubic graph $G$, there exists a perfect matching containing no directed cut, our algorithm proceeds by iterating through all perfect matchings of $G$. For each of them, all minimal cuts will be chosen consecutively to be directed and the algorithm will recurse. This exploration will either lead to a conflict, i.e., at some step of the algorithm the current perfect matching will contain only cuts having arcs into both directions; or will reduce the set of remaining perfect matchings to consider (since some of their cuts are already directed).

In fact, we observed that many perfect matchings in cubic graphs contain only one single cut and this speeds up the process, since no branching is needed for the choice of which cut to direct.

By the discussion in Section~\ref{sec:construction}, one can observe that for a given cubic graph $G=(V,E)$, it is enough to check whether an arbitrarily chosen edge $\{u,v\}$ of $G$ can be an $a$-arc i.e. the edges of every perfect matching $M$ containing $\{u,v\}$, can be oriented such that $M$ contains a directed cut. Since $G$ is cubic, the set of perfect matchings containing $\{u,v\}$ is exactly the set of perfect matchings of $G'=(V-\{u,v\},E)$. If $G$ cannot be oriented such that $\{u,v\}$ is an $a$-arc, then $G$ cannot be a counterexample. This allows a more efficient elimination of graph candidates.

The program was written in a Python-based language using tools provided by SageMath~\cite{Sage}.\footnote{The source code can be provided by the authors on request and is available
as supplementary material on the publisher's site.}

\bigskip

We verified that all cubic, 3-connected digraphs with at most 22 vertices have a perfect matching without a directed cut. We also checked that the only partially directed graph on 24 vertices admitting an orientation such that every perfect matching contains a directed cut, is the one of Figure~\ref{subfig:smallestcount}. Therefore, this is the unique  smallest counterexample to Conjecture~\ref{conj:Hochstattler'}. 

\bigskip

A counterexample to Conjecture~\ref{conj:Neumann-Lara'} has to be an orientation of a counterexample to Conjecture~\ref{conj:Tait}. The list of all triangle-free counterexamples to Conjecture~\ref{conj:Tait} (up to a certain order) can be found on McKay's webpage~\cite{McKayWeb}, see Table~\ref{tab:NeumannLara}. We use the list as a base for our computation. By item 1 of Lemma~\ref{lem:girth4}, our experimental results show that Conjecture~\ref{conj:Neumann-Lara'} holds for planar cubic, 3-connected graphs on at most 48 vertices. In particular, Conjecture~\ref{conj:Hochstattler} holds for planar cubic graphs on at most 48 vertices. Recall that by Euler's formula, a planar cubic graph on $n$ vertices has $2+\frac{n}{2}$ faces. Now, by planar duality and Lemma~\ref{lem:duality}, we conclude that Conjecture~\ref{conj:Neumann-Lara} holds for planar triangulations on at most $2+\frac{48}{2}=26$ vertices. Since, removing edges does not create directed cycles, we conclude:

\begin{proposition}\label{prop:NeumannLara26}
Conjecture~\ref{conj:Neumann-Lara} holds for planar graphs on at most 26 vertices. 
\end{proposition} 

Moreover, from columns 3, 4 and 5 of Table~\ref{tab:NeumannLara}, we can deduce that Conjecture~\ref{conj:Neumann-Lara} holds for 4-vertex-connected planar graphs on at most 27 vertices and for 4-vertex connected planar graphs with minimum degree 5 on at most 39 vertices. 

\afterpage{
\begin{table}[!ht]
\centering
	\begin{footnotesize}
	\begin{tabular}{|c|c|c|c|c|}
	 \hline
	\textbf{Nr of} & \textbf{No} & \textbf{Cyclically} & \textbf{no faces of size 3,4},& \textbf{Cyclically} \\
	 \textbf{vertices} & \textbf{restriction} & \textbf{4-connected} & \textbf{cyclic connectivity exactly 4}  & \textbf{5-connected} \\
	\hline
	\hline
	$\leq 36$ & 0 & 0 & 0 & 0\\
	$38$ & 6 & 0 & 0 & 0 \\
	$40$ & 37 & 0 & 0 & 0 \\
	$42$ & 277 & 3 & 0 & 0 \\
	$44$ & 1732 & 3 & 1 & 1 \\
	$46$ & 11204 & 19 & 3 & 1 \\
	$48$ & 70614 & 30 & 1 & 0 \\
	$50$ & ? & 126 & 3 & 3 \\
	$52$ & ? & ? & 6 & 6 \\
	$54$ & ? & ? & 12 & 2 \\
	$56$ & ? & ? & 49 & 22 \\
	$58$ & ? & ? & 126 & 37 \\
	$60$ & ? & ? & 214 & 31 \\
	$62$ & ? & ? & 659 & 194\\
	$64$ & ? & ? & 1467 & 298 \\
	$66$ & ? & ? & 3247 & 306 \\
	$68$ & ? & ? & 9187 & 1538 \\
	$70$ & ? & ? & 22069 & 2566 \\
	$72$ & ? & ? & 50514 & 3091 \\
	$74$ & ? & ? & 137787 & 13487 \\
	$76$ & ? & ? & 339804 & ? \\
	\hline
	\hline 
	\end{tabular}
	\end{footnotesize}
\caption[Webpage of B. McKay]{The counterexamples to Tait's conjecture from~\cite{McKayWeb} on which Conjecture~\ref{conj:Neumann-Lara'} was verified.}
\label{tab:NeumannLara}
\end{table}

}

\bigskip

Recall that if $G$ admits a partial orientation $D$ with an $a$-arc, then $D$ can be used as a gadget described in Section~\ref{sec:construction} for building counterexamples.
Also recall that the construction of the graph $\widetilde{D}$ (Figure~\ref{subfig:a-arcs-constructions3}) preserves planarity. Thus, we get:

\begin{observation}
If Conjecture~\ref{conj:Neumann-Lara} is true, then no planar, 3-connected, cubic digraph has an $a$-arc.
\end{observation}

Hence, one way to build a counterexample to Conjecture~\ref{conj:Neumann-Lara} is to find a planar, 3-connected, cubic digraph having an $a$-arc. Since an $a$-arc corresponds to an $a$-edge in the underlying undirected graph, we considered planar cubic graphs having an $a$-edge. Moreover, according to item 2 of Lemma~\ref{lem:girth4}, only graphs of girth 4 need to be considered. The list of such graphs of order at most 38 was generated by McKay and provided on his webpage~\cite{McKayWeb}. Our computations show that none of them can be oriented such that it has an $a$-arc.

\section{Conclusion}
\label{sec:concl}
We have disproved Conjecture~\ref{conj:Hochstattler}. For its variant, Conjecture~\ref{conj:Hochstattler'}, we found the unique smallest counterexample on 24 vertices. However, our counterexample to Conjecture~\ref{conj:Hochstattler} is quite large. Is there a counterexample to Conjecture~\ref{conj:Hochstattler} on less than 122 vertices? Similarly for Conjecture~\ref{conj:Tutte}, the smallest known counterexample is on 50 vertices and we have confirmed that this conjecture is true for graphs on at most 40 vertices. What is the size of a smallest counterexample to Conjecture~\ref{conj:Tutte} ? 

As a common weakening of Conjectures~\ref{conj:Hochstattler'} and~\ref{conj:Tutte} we proposed Conjecture~\ref{conj:KV}: \emph{Every 3-connected, cubic, bipartite digraph  contains a perfect matching without directed cut.} Since Conjecture~\ref{conj:KV} is weaker than Conjecture~\ref{conj:Tutte}, we know that it holds for all graphs on at most 40 vertices. Furthermore, we verified by computer that no orientation of the known\footnote{see \url{http://mathworld.wolfram.com/TutteConjecture.html}} small counterexamples to Conjecture~\ref{conj:Tutte} is a counterexample to Conjecture~\ref{conj:KV}.

Note that the weakest statement in Figure~\ref{fig:diagram} (corresponding to the minimum in the diagram) is a natural open case of Conjecture~\ref{conj:Neumann-Lara}: \emph{Every orientation of a planar Eulerian triangulation can be vertex-partitioned into two acyclic sets.}

A well-known parameter of cubic graphs is \emph{cyclic connectivity}, i.e., the size of a smallest \emph{cyclic cut}. A cyclic cut is a minimal edge-cut, such that after its deletion both components have a cycle. Any minimum cyclic cut forms a matching and all edge-cuts that are matchings are cyclic cuts. Clearly, if all cyclic cuts are large, then it is harder to have a cut in every perfect matching. That is why it is natural to consider any of the conjectures on cubic graphs in this paper for high cyclic connectivity.

While Conjecture~\ref{conj:Tait} is false for all possible values of cyclic connectivity (planar graphs have cyclic connectivity at most $5$), the highest cyclic connectivity among the known counterexamples to Conjectures~\ref{conj:Tutte} and~\ref{conj:Hochstattler'} is $4$ and $3$, respectively.
Indeed, it is not even known whether there exist general non-Hamiltonian, 3-connected, cubic graphs of high cyclic connectivity. The existence of such graphs is conjectured in~\cite{H14} while according to~\cite{MM16}, Thomassen conjectured the contrary (in personal communication in 1991). The 28-vertex Coxeter graph has cyclic connectivity $7$, witnessing the largest known cyclic connectivity among non-Hamiltonian cubic 3-connected graphs. We checked that to attain cyclic connectivity 8 and 9, at least 48 and 66 vertices are needed, respectively.

A (probably novel) parameter that is still closer to these problems is the size of a smallest edge-cut contained in a perfect matching. We believe that it deserves further investigation.

\subsubsection*{Acknowledgements} We thank Winfried Hochst\"attler and Raphael Steiner for reading and commenting an earlier version of this manuscript. We also thank Brendan McKay for making available on-line some data on planar graphs on our request. We thank the reviewers for their careful reading and valuable suggestions which helped improving this paper.
The first author was supported by ANR GATO ANR-16-CE40-0009-01.

\end{document}